\documentclass{amsart}

\usepackage[left=1.5in, right=1.5in]{geometry}     
\usepackage{graphicx}              
\usepackage{amsmath}               
\usepackage{amsfonts}              
\usepackage{amsthm}                
\usepackage{enumitem}
\usepackage{amssymb}
\usepackage{url}
\usepackage{color}
\usepackage{mathrsfs}
\usepackage{mathtools}
\usepackage{mathabx}

\numberwithin{equation}{section}

\newtheorem{thm}{Theorem}[section]
\newtheorem{lem}[thm]{Lemma}
\newtheorem{prop}[thm]{Proposition}
\newtheorem{cor}[thm]{Corollary}

\theoremstyle{definition}
\theoremstyle{definition}
\theoremstyle{definition}\newtheorem*{remark}{Remark}

\newcommand{\norm}[1]{\left\lVert#1\right\rVert} 
\newcommand{\RR}{\mathbb{R}}            

\newcommand{\CC}{\mathbb{C}}
\newcommand{\ZZ}{\mathbb{Z}}

\newcommand{\Intr}{\displaystyle\int}
\newcommand{\Sumn}{\displaystyle\sum}
\newcommand{\Suma}{\Sumn_{\alpha\in R_+}}
\newcommand*\diff{\mathop{}\!\mathrm{d}}

\newcommand{\al}{\alpha}
\newcommand{\alx}{\langle\alpha,x\rangle}

\newcommand{\DT}{\mathcal{D}_k}

\newcommand{\IntN}{\displaystyle\int_{\RR^N}}

\begin{document}

\title[Hardy-type inequalities for Dunkl operators]{Hardy-type inequalities for Dunkl operators with applications to many-particle Hardy inequalities}
\author{Andrei Velicu}
\address{Andrei Velicu, Department of Mathematics, Imperial College London, Huxley Building, 180 Queen's Gate, London SW7 2AZ, UK}
\email{a.velicu15@imperial.ac.uk}
\date{}


\begin{abstract}
In this paper we study various forms of the Hardy inequality for Dunkl operators, including the classical inequality, $L^p$ inequalities, an improved Hardy inequality, as well as the Rellich inequality and a special case of the Caffarelli-Kohn-Nirenberg inequality. As a consequence, one-dimensional many-particle Hardy inequalities for generalised root systems are proved, which in the particular case of root systems $A_{N-1}$ improve some well-known results.
\end{abstract}

\maketitle

\section{Introduction}

The classical Hardy inequality
$$ \IntN |\nabla f|^2 \diff x \geq \frac{(N-2)^2}{4} \IntN \frac{f^2}{|x|^2} \diff x,$$
which holds for all $f\in C_c^\infty(\RR^N)$ if $N\geq 2$, and for all $f\in C_c^\infty (\RR \setminus \{0\})$ if $N=1$, is one of the most important results in analysis. It has seen an incredible development from its beginnings in Hardy's papers, having been refined and extended to various settings; see \cite{BEL}, \cite{OK}, \cite{Davies1998} and references therein for an overview of the topic. 

Here we begin a systematic study of Hardy's inequality and its variants for Dunkl operators. Before we present the content of this paper, let us mention some related results in the literature. Some Hardy-type inequalities for function spaces associated to Dunkl operators were proved in \cite{Mejjaoli1} and \cite{Mejjaoli2}; in particular, the $L^2$ Hardy inequality and Hardy inequality for fractional Dunkl Laplacian were obtained (without any estimates on the constants).  In \cite{CRT}, the ground state representation method was used to obtain Hardy's inequality for fractional powers of the Dunkl-Hermite operators.  In \cite{GIT}, the authors proved a Pitt's inequality which implies the Hardy's inequality for fractional powers of the Dunkl Laplacian, while in \cite{GIT17} the same authors proved a Stein-Weiss inequality which in turn implies a sharp $L^p$ Hardy inequality on a Lizorkin space $\Phi_k$. 

Apart from the classical Hardy's inequality with sharp constant, we also study in this paper $L^p$ inequalities which hold for small coefficients $\gamma$ and Hardy inequality for fractional Laplacian. As in the classical case, the sharp constant in the Hardy inequality is not achieved and we prove here an improved Hardy inequality using a method based on spherical h-harmonics. Two Hardy-type inequalities are also discussed, the Rellich inequality (with sharp constant), and a special case of the Caffarelli-Kohn-Nirenberg inequality. More background on these variants is given in each corresponding section. 

In the last part of the paper we provide an application of the Hardy inequality for Dunkl operators to many-particle Hardy inequalities in dimension $d=1$ associated to root systems. Many-particle Hardy inequalities are generalisations of the classical Hardy inequality for $N$ particles in $\RR^d$ involving mutual distances between the particles. The standard example of such an inequality is 
\begin{equation} \label{standardmanyparticle} 
\sum_{i=1}^N \int_{\RR^{dN}} |\nabla_i f|^2 \diff \omega \geq C(d,N)  \int_{\RR^{dN}} |f|^2 \sum_{1 \leq i < j \leq N} \frac{1}{r_{ij}^2} \diff\omega,
\end{equation}
which was proved in \cite{HOHOLT} for $r_{ij}=|\omega_i-\omega_j|$. Here $\omega=(\omega_1, \ldots, \omega_N) \in \RR^{dN}$ so each $\omega_i \in \RR^d$, and the inequality holds for all $f\in C_c^\infty(\RR^{dN})$ if $d \geq 3$, and for all $f\in C_c^\infty (\RR^N \setminus \bigcup_{1\leq i<j\leq N} \{x : x_i=x_j\})$ if $d=1$. Similar results but for different $r_{ij}$, or with different types of geometric couplings, have also been proved in \cite{LundholmSolovej} and \cite{Lundholm}.

One of the main applications of Dunkl theory is in the mathematical physics of quantum many body problems, especially in the study of the Calogero-Moser-Sutherland (CMS) models. The linear CMS model (see \cite{Calogero}, \cite{Sutherland}) describes the interactions of $N$ particles on the real line with potentials of inverse square type, and it is characterised by the Hamiltonian
\begin{equation} \label{CMSHamiltonian}
\mathcal{H} := -\Delta + g \sum_{1\leq i < j \leq N} \frac{1}{(x_i-x_j)^2}.
\end{equation}
A new understanding of the quantum integrability of this system was possible by considering a modified version of the Hamiltonian \eqref{CMSHamiltonian} involving reflection terms, which can be expressed in terms of Dunkl operators with root system of type $A_{N-1}$ (generalised CMS models for other root systems were also considered in \cite{OP76} and \cite{OP78}). For more information about CMS models see the monograph \cite{vDV}, and for their connection to Dunkl operators see \cite{Rosler} and references therein. 

The Hardy inequality associated to the CMS model with Hamiltonian \eqref{CMSHamiltonian} is
\begin{equation} \label{CMSHardy} 
\IntN |\nabla f|^2 \diff x \geq \frac{1}{2} \IntN |f|^2 \sum_{1\leq i < j \leq N} \frac{1}{(x_i-x_j)^2} \diff x,
\end{equation}
which is valid for all $f\in C_c^\infty (\RR^N \setminus \bigcup_{1\leq i<j\leq N} \{x : x_i=x_j\})$, and where the constant $\frac{1}{2}$ is sharp. This inequality corresponds precisely to \eqref{standardmanyparticle} in the case $d=1$. We exploit here the connection between Dunkl operators and (generalised) CMS models to obtain new improved one-dimensional many-particle Hardy inequalities associated to general root systems. For example, in the particular case of root system $A_{N-1}$ we obtain
$$ \IntN |\nabla f|^2 \diff x 
\geq \frac{1}{2} \IntN |f|^2 \sum_{1\leq i < j \leq N} \frac{1}{(x_i-x_j)^2} \diff x
+ \frac{(N^2+N-4)^2}{16} \IntN \frac{|f|^2}{|x|^2} \diff x,$$
valid for $f\in C_c^\infty (\RR^N \setminus \bigcup_{i\neq j} \{x: x_i = x_j\} \setminus \bigcup_i \{x: x_i=0\})$; this result generalises \eqref{standardmanyparticle} (and also \eqref{CMSHardy}).

The paper is organised as follows. In section 2 we briefly introduce the theory of Dunkl operators and spherical h-harmonics. In section 3 we discuss the $L^p$ Hardy inequality; for general $p$ we obtain this inequality for small $\gamma$, while in the case $p=2$ we obtain the $L^2$ Hardy inequality for the full range of $\gamma$. In section 4 we prove the Hardy inequality for fractional Dunkl Laplacian. In section 5 we prove an improved Hardy inequality and as a corollary we deduce the Poincar\'e inequality for Dunkl operators. Section 6 contains two Hardy-type results: the Rellich inequality and the Caffarelli-Kohn-Nirenberg inequality. Finally, in section 7 we prove many-particle Hardy inequalities in dimension $1$ associated to general root systems.

\section{Preliminaries}

In this section we will present a very quick introduction to Dunkl operators. For more details see the survey papers \cite{Rosler} and \cite{Anker}.

The function spaces under consideration in this paper are generally real-valued. However, Dunkl operators can be defined just as well on complex-valued functions and the basic theory holds also in this setting. In some of the inequalities below (particularly when relevant in applications), we suggest extensions of our results to complex-valued functions, and to make the distinction in this case we write for example $C^1(\RR^N;\CC)$, $C_c^\infty(\RR^N;\CC)$, etc. 
 
A root system is a finite set $R\subset \RR^N\setminus \{0\}$ such that $R \cap \alpha \RR = \{ -\alpha, \alpha\}$ and $\sigma_\alpha(R) = R$ for all $\alpha\in R$. Here $\sigma_\alpha$ is the reflection in the hyperplane orthogonal to the root $\alpha$, i.e.,
$$ \sigma_\alpha x = x - 2 \frac{\alx}{\langle \alpha,\alpha \rangle} \alpha.$$
The group generated by all the reflections $\sigma_\alpha$ for $\alpha\in R$ is a finite group, and we denote it by $G$. 

The Weyl chambers associated to the root system $R$ are the connected components of $\RR^N \setminus \{x\in\RR^N : \alx =0 \text{ for some } \alpha\in R \}$. It can be checked that the reflection group $G$ acts simply transitively on the set of Weyl chambers so, in particular, the number of Weyl chambers equals the order of the group, $|G|$.

Let $k:R \to [0,\infty)$ be a $G$-invariant function, i.e., $k(\alpha)=k(g\alpha)$ for all $g\in G$ and all $\alpha\in R$. We will normally write $k_\alpha=k(\alpha)$ as these will be the coefficients in our Dunkl operators. Let us note that Dunkl operators can be defined more generally for $k:R \to \CC$, but much of the theory presented in this section requires our stronger assumption $k\geq 0$. We can write the root system $R$ as a disjoint union $R=R_+\cup (-R_+)$, and we call $R_+$ a positive subsystem;  this decomposition is not unique, but the particular choice of positive subsystem does not make a difference in the definitions below because of the $G$-invariance of the coefficients $k$.

From now on we fix a root system in $\RR^N$ with positive subsystem $R_+$. We also assume without loss of generality that $|\alpha|^2=2$ for all $\al
\in R$. For $i=1,\ldots, N$ we define the Dunkl operator on $C^1(\RR^N)$ by
$$ T_i f(x) = \partial_i f(x) + \Suma k_\alpha \alpha_i \frac{f(x)-f(\sigma_\alpha x)}{\alx}.$$
We will denote by $\nabla_k=(T_1,\ldots, T_N)$ the Dunkl gradient, and $\Delta_k = \displaystyle\sum_{i=1}^N T_i^2$ will denote the Dunkl Laplacian. Note that for $k=0$ Dunkl operators reduce to partial derivatives, and $\nabla_0=\nabla$ and $\Delta_0=\Delta$ are the usual gradient and Laplacian.

We can express the Dunkl Laplacian in terms of the usual gradient and Laplacian using the following formula:
\begin{equation} \label{Dunkllaplacian}
\Delta_k f(x) = \Delta f(x) + 2\Suma k_\alpha \left[ \frac{\langle \nabla f(x),\alpha \rangle}{\alx} - \frac{f(x)-f(\sigma_\alpha x)}{\alx^2} \right].
\end{equation}

The weight function naturally associated to Dunkl operators is
$$ w_k(x) = \prod_{\alpha\in R_+} |\alx|^{2k_\alpha}.$$
This is a homogeneous function of degree $2\gamma$, where
$$ \gamma := \Suma k_\alpha.$$
We will work in spaces $L^p(\mu_k)$, where $\diff\mu_k = w_k(x) \diff x$ is the weighted measure; the norm of these spaces will be written simply $\norm{\cdot}_p$. With respect to this weighted measure we have the integration by parts formula
$$ \IntN T_i(f) g \diff\mu_k = - \IntN f T_i(g) \diff\mu_k.$$

If one of the functions $f,g$ is $G$-invariant, then we have the Leibniz rule
$$ T_i(fg) = f T_ig + g T_if.$$
In general we have
$$ T_i(fg)(x) = T_if(x)g(x) + f(x)T_ig(x) - \Suma k_\alpha \alpha_i \frac{(f(x)-f(\sigma_\alpha x))(g(x)-g(\sigma_\alpha x))}{\alx}.$$

An important Dirichlet form estimate for the weighted norms $L^2$ of the Dunkl gradient and the classical gradient was proved in \cite{V} using carr\'e-du-champ methods:

\begin{lem} [\cite{V}] \label{dunkl-euclidean-dirichletform}
For all $f\in C_0^1(\RR^N)$ we have
$$ \IntN |\nabla_k f|^2 \diff\mu_k \geq \IntN |\nabla f|^2 \diff\mu_k.$$
\end{lem}

A Sobolev inequality is available for the Dunkl gradient (see \cite{V}):

\begin{prop}
Let $1 \leq p < N+2\gamma$ and $q=\frac{p(N+2\gamma)}{N+2\gamma-p}$. Then there exists a constant $C>0$ such that we have the inequality 
$$ \norm{f}_q \leq C \norm{\nabla_k f}_p \qquad \forall f\in C_c^\infty(\RR^N).$$
\end{prop}

An important function associated with Dunkl operators is the Dunkl kernel $E_k(x,y)$, defined on $\CC^N\times\CC^N$, which acts as a generalisation of the exponential and is defined, for fixed $y\in\CC^N$, as the unique solution $Y=E_k(\cdot, y)$ of the equations
$$ T_iY=y_i Y, \qquad i=1,\ldots N,$$
which is real analytic on $\RR^N$ and satisfies $Y(0)=1$. 

Another definition of the Dunkl exponential can be given in terms of the intertwining operator $V_k$ which is initially defined on the space $\CC[\RR^N]$ of $\CC$-polynomials on $\RR^N$ as the unique linear operator $V_k : \CC[\RR^N] \to \CC[\RR^N]$ which satisfies the properties

\begin{itemize}

\item $V_k(c)=c$ for all $c\in\CC$

\item $V_k(\CC[\RR^N]_n)=\CC[\RR^N]_n$ for all $n\geq 1$, where $\CC[\RR^N]_n$ is the space of homogeneous $n$-degree polynomials

\item $T_i V_k = V_k \partial_i$.
\end{itemize}
The last property is the characteristic intertwining formula which connects Dunkl operators to partial derivatives. The construction of this operator is done using simple linear algebra. It is then possible to extend it to a larger space of functions which in particular contains the exponentials (see \cite[Section 2.4]{Rosler} for more details) and so the Dunkl exponential can then be equivalently defined as
$$ E_k(x,y) = V_k \left( e^{\langle \cdot, y \rangle} \right) (x).$$

The following growth estimates on $E_k$ are known: for all $x\in\RR^N$, $y\in\CC^N$ and all $\beta\in\ZZ^N_+$ we have
$$ |\partial_y^\beta E_k(x,y)| \leq |x|^{|\beta|} \displaystyle\max_{g\in G} e^{\text{Re}\langle gx,y\rangle}.$$

It is then possible to define a Dunkl transform on $L^1(\mu_k)$ by
$$ \mathcal{D}_k(f)(\xi)= \frac{1}{M_k}\Intr_{\RR^N} f(x)E_k(-i\xi,x)\diff \mu_k(x), \qquad \text{ for all } \xi\in\RR^N,$$
where 
$$ M_k = \IntN e^{-|x|^2/2} \diff\mu_k(x)$$
is the Macdonald-Mehta integral. The Dunkl transform extends to an isometric isomorphism of $L^2(\mu_k)$; in particular, the Plancherel formula holds. When $k=0$ the Dunkl transform reduces to the Fourier transform.

\subsection*{Spherical h-harmonics}

We will briefly introduce h-harmonics; our presentation here is based on \cite{DX} and we invite the interested reader to this reference for more details. An $h$-harmonic polynomial of degree $n$ is a homogeneous polynomial $p$ of degree $n$ that satisfies
$$ \Delta_k p=0.$$
Spherical $h$-harmonics (or just $h$-harmonics) of degree $n$ are then defined to be restrictions of $h$-harmonic polynomials of degree $n$ to the sphere $\mathbb{S}^{N-1}$. Let $\mathcal{H}_n^N$ be the space of $h$-harmonics of degree $n$; this is a finite-dimensional space and denote its dimension by $d(n)$. Moreover, the space $L^2(\mathbb{S}^{N-1}, w_k(\xi)\diff\xi)$ is the orthogonal direct sum of the spaces $\mathcal{H}_n^N$, for $n=0,1,2,\ldots$. Let 
$$Y_i^n \quad\text{ for } i=1,\ldots, d(n)$$
be an orthonormal basis of $\mathcal{H}_n^N$. In spherical polar coordinates $x=r\xi$, for $r\in [0,\infty)$ and $\xi\in\mathbb{S}^{N-1}$, we can write the Dunkl Laplacian as
\begin{equation} \label{hharmonicslaplacian} 
\Delta_k=\frac{\partial^2}{\partial r^2}+ \frac{N+2\gamma-1}{r}\frac{\partial}{\partial r}+\frac{1}{r^2}\Delta_{k,0},
\end{equation}
where $\Delta_{k,0}$ is a generalisation of the Laplace-Beltrami operator on the sphere, and it only acts on the $\xi$ variable. Then the spherical harmonics are eigenvalues of $\Delta_{k,0}$, i.e.,
$$ \Delta_{k,0} Y= -n(n+N+2\gamma-2)Y=:\lambda_nY, \quad\text{ for all } Y\in\mathcal{H}_n^N.$$
The $h$-harmonic expansion of a function $f\in L^2(\mu_k)$ is given by
$$ f(r\xi) 
= \Sumn_{n=0}^\infty \Sumn_{i=1}^{d(n)} f_{n,i}(r)Y^n_i(\xi),$$
where
\begin{equation} \label{hharmonicsintegral} 
f_{n,i}(r)=\Intr_{\mathbb{S}^{N-1}} f(r\xi) Y_i^n(\xi) w_k(\xi) \diff\sigma(\xi),
\end{equation}
and $\sigma$ is the surface measure on the sphere $\mathbb{S}^{N-1}$.

\section{$L^p$ Hardy inequality} \label{SEC:Lphardy}

In this section we study general $L^p$ Hardy inequalities for Dunkl operators. Firstly we prove two Hardy-type inequalities in the weighted spaces $L^p(\mu_k)$, for the usual gradient and the Dunkl gradient. Using this inequality, we obtain the $L^p$ inequality for a restricted range of $p$. In the particular case $p=2$ we are able to improve on this method by using a stronger estimate, and so we finally prove the $L^2$ Hardy inequality for Dunkl operators in full generality. 

We begin with the following Hardy-type inequalities.

\begin{thm} \label{hardytype}
Let $1 < p < \infty$. Then, for any $f\in C_c^\infty(\RR^N)$ we have the inequality
\begin{equation} \label{hardytype1} 
\IntN |\langle x, \nabla f \rangle |^p \diff\mu_k \geq \left( \frac{N+2\gamma}{p} \right)^p \IntN |f|^p \diff\mu_k.
\end{equation}
If, in addition, we have $p<\frac{N+2\gamma}{2\gamma}$, then we also have
\begin{equation} \label{hardytype2} \IntN |\langle x, \nabla_k f \rangle|^p \diff\mu_k \geq \left( \frac{N+2\gamma-2p\gamma}{p} \right)^p \IntN |f|^p \diff\mu_k
\end{equation}
for all $f\in C_c^\infty(\RR^N)$.
\end{thm}

\begin{proof}
We first note that 
$$ \sum_{i=1}^N T_i(x_i) 
= \sum_{i=1}^N \left(1+ \Suma k_\alpha\alpha_i^2 \right)
=N+2\gamma.$$
Thus, we have
\begin{align*}
(N+2\gamma) \IntN |f|^p \diff\mu_k 
&= \sum_{i=1}^N \IntN T_i(x_i) |f|^p \diff\mu_k 
\\
&= - \sum_{i=1}^N \IntN x_i T_i(|f|^p) \diff\mu_k
\\
&= - \IntN \left\langle x,  \nabla(|f|^p) +  \Suma k_\alpha \alpha \frac{|f(x)|^p-|f(\sigma_\alpha x)|^p}{\alx} \right\rangle \diff\mu_k
\\
&= - p \IntN f|f|^{p-2} \langle x, \nabla f\rangle \diff\mu_k - \Suma k_\alpha \IntN \left(|f(x)|^p-|f(\sigma_\alpha x)|^p \right) \diff\mu_k
\end{align*}

Note that by a change of variables $y=\sigma_\alpha x$, we have
$$\IntN |f(\sigma_\alpha x)|^p \diff\mu_k(x) = \IntN |f(y)|^p \diff\mu_k(y),$$
so
\begin{equation} \label{hardytypestep} 
(N+2\gamma) \IntN |f|^p \diff\mu_k 
= - p \IntN f|f|^{p-2} \langle x, \nabla f \rangle \diff\mu_k.
\end{equation}

By H\"older's inequality with exponents $\frac{p}{p-1}$ and $p$, we have
$$ - \IntN f|f|^{p-2} \langle x, \nabla f \rangle \diff\mu_k
\leq \left(\IntN |f|^p \diff\mu_k\right)^\frac{p-1}{p} \left( \IntN |\langle x, \nabla f \rangle|^p \diff\mu_k \right)^\frac{1}{p}.$$
Combining this inequality with (\ref{hardytypestep}), we obtain (\ref{hardytype1}). 

On the other hand, going back to (\ref{hardytypestep}), we have
\begin{align}
(N+2\gamma) \IntN |f|^p \diff\mu_k 
&= - p \IntN f|f|^{p-2} \langle x , \nabla_k f \rangle \diff\mu_k \nonumber
\\
&\qquad \qquad 
+ p \Suma k_\alpha \IntN f|f|^{p-2} (f(x)-f(\sigma_\alpha x)) \diff\mu_k \label{hardydifferenceterm}
\\
&= - p \IntN f|f|^{p-2} \langle x, \nabla_k f \rangle \diff\mu_k + p\gamma \IntN |f|^p \diff\mu_k \nonumber
\\
&\qquad \qquad 
-p\Suma k_\alpha \IntN f(x)f(\sigma_\alpha x)|f(x)|^{p-2} \diff\mu_k. \nonumber
\end{align}
By H\"older's inequality with exponents $\frac{p}{p-1}$ and $p$, we have
\begin{align*}
- \IntN f(x)f(\sigma_\alpha x)|f(x)|^{p-2} \diff\mu_k 
&\leq \left(\IntN |f(x)|^p \diff\mu_k \right)^\frac{p-1}{p} \left( \IntN |f(\sigma_\alpha x)|^p \diff\mu_k \right)^\frac{1}{p}
\\
&=\IntN |f|^p \diff\mu_k,
\end{align*}
where in the last step we used a change of variables. Thus, going back to the computations above, we now have
$$ (N+2\gamma -2p\gamma) \IntN |f|^p \diff\mu_k \leq - p \IntN f|f|^{p-2} \langle x , \nabla_k f \rangle \diff\mu_k$$
and as before, using H\"older's inequality, we obtain (\ref{hardytype2}).
\end{proof}

From this inequality we easily obtain an $L^p$ Hardy inequality for a restricted range of $p$.

%%%%%%%%%%%%%%%%%%%%%%%%%%%%%%%%%%%%%%%%%%%%%%%%%%%%%%%%%%%%%%%%%%%%%%%%%%%%%%%%%%%%%%%%%%%%%%%%%
\begin{thm}
Let $1<p<\frac{N+2\gamma}{1+2\gamma}$. Then, for any $f \in C_c^\infty(\RR^N)$, we have the inequality
$$ \IntN |\nabla_k f|^p \diff\mu_k \geq \left( \frac{N+2\gamma-2p\gamma-p}{p}\right)^p \IntN \frac{|f|^p}{|x|^p} \diff\mu_k.$$
\end{thm}
%%%%%%%%%%%%%%%%%%%%%%%%%%%%%%%%%%%%%%%%%%%%%%%%%%%%%%%%%%%%%%%%%%%%%%%%%%%%%%%%%%%%%%%%%%%%%%%%%

\begin{proof}
This follows by taking $\frac{f}{|x|}$ in (\ref{hardytype2}). Indeed, we have
\begin{equation} \label{lphardystep1} 
\frac{N+2\gamma-2p\gamma}{p} \norm{\frac{f}{|x|}}_p 
\leq \norm{\left\langle x, \nabla_k  \left( \frac{f}{|x|} \right) \right\rangle}_p.
\end{equation}
By the Leibniz rule we have
$$ \nabla_k \left( \frac{f}{|x|}\right) = \frac{1}{|x|} \nabla_k f - f \frac{x}{|x|^3},$$
so
\begin{equation} \label{lphardystep2} 
\norm{\left\langle x, \nabla_k \left( \frac{f}{|x|}\right) \right\rangle}_p 
\leq \norm{\nabla_k f}_p + \norm{\frac{f}{|x|}}_p.
\end{equation}
Thus, combining (\ref{lphardystep1}) and (\ref{lphardystep2}), we obtain the desired $L^p$ Hardy inequality.
\end{proof}

We obtained this $L^p$ Hardy inequality through the Hardy-type inequality (\ref{hardytype2}), which in turn relies on an estimate based on the H\"older's inequality of the term containing the difference $f(x)-f(\sigma_\alpha x)$ in (\ref{hardydifferenceterm}). A better result can be obtained in the case $p=2$ by relying instead on the Hardy-type inequality (\ref{hardytype1}) for the usual gradient, and the Dirichlet form estimate 
$$ \norm{\nabla f}_2 \leq \norm{\nabla_k f}_2$$
from Lemma \ref{dunkl-euclidean-dirichletform}. Using this approach, we can obtain the $L^2$ Hardy inequality without any restrictions on $\gamma$.

\begin{thm} \label{hardy}
Assume $N+2\gamma > 2$. Then, for any $f\in C_c^\infty(\RR^N; \CC)$, we have the inequality
$$ \IntN |\nabla_k f|^2 \diff\mu_k \geq \left( \frac{N+2\gamma-2}{2} \right)^2 \IntN \frac{|f|^2}{|x|^2} \diff\mu_k,$$
where the constant is sharp.
\end{thm}

\begin{proof} 
Firstly, we notice that it is enough to consider real-valued functions. Indeed, if the result holds for real-valued functions and $f\in C_c^\infty(\RR^N;\CC)$, then applying the inequality to $\text{Re} f$ and $\text{Im} f$ and adding up the two resulting inequalities, we obtain that the conclusion holds true also for $f$.

We proceed similarly to the proof of the previous Theorem, but now we take $\frac{f}{|x|}$ in inequality (\ref{hardytype1}) instead. We have
$$ \frac{N+2\gamma}{2} \norm{\frac{f}{|x|}}_2 
\leq \norm{\left\langle x, \nabla \left( \frac{f}{|x|}\right) \right\rangle}_2
\leq \norm{\nabla f}_2 + \norm{\frac{f}{|x|}}_2. $$
Thus
$$ \frac{N+2\gamma-2}{2} \norm{\frac{f}{|x|}}_2 \leq \norm{\nabla f}_2 \leq \norm{\nabla_k f}_2,$$
where we used Lemma \ref{dunkl-euclidean-dirichletform}.

To check that the constant $\frac{(N+2\gamma-2)^2}{4}$ is sharp we consider, for each $n=1, 2,\ldots$, the radial function $f_n(x)=h_n(|x|)$, where
$$ h_n(r) = 
\begin{cases}
\frac{1}{c_n} & \text{ if } r\leq 1 \\
\frac{1}{c_n}r^{c_n} & \text{ if } r>1,
\end{cases}$$
where $c_n = -\frac{1}{n} - \frac{N+2\gamma-2}{2}$. Then we have
\begin{align*}
\frac{\IntN |\nabla_k f_n|^2 \diff\mu_k}{\IntN \frac{f_n(x)^2}{|x|^2}\diff\mu_k}
= \frac{\displaystyle\int_0^\infty (h'_n(r))^2 r^{N+2\gamma-1} \diff r}{\displaystyle\int_0^\infty h_n(r)^2 r^{N+2\gamma-3} \diff r}
=\frac{\frac{n}{2}}{\frac{1}{c_n^2}\left(\frac{1}{N+2\gamma} + \frac{n}{2} \right)} 
\to \frac{(N+2\gamma-2)^2}{4}
\end{align*}
as  $n\to \infty$.
\end{proof}

\begin{remark}
As mentioned above, we were able to prove the Hardy inequality in case $p=2$ for the full range $N+2\gamma >2$ because we replaced a H\"older inequality estimate for difference terms by the stronger estimate of Lemma \ref{dunkl-euclidean-dirichletform}. If we have in general 
$$ \norm{\nabla f}_p \leq \norm{\nabla_k f}_p,$$
then the same method applies to obtain $L^p$ Hardy inequality in full generality for $1<p<N+2\gamma$.
\end{remark}

\section{Hardy inequality for fractional Laplacian} \label{SEC:fractionalhardy}

In this short section we prove the Hardy inequality for fractional Laplacian using the Pitt's inequality. In particular, we obtain a new proof of the $L^2$ Hardy inequality from Theorem \ref{hardy}.

We can define the fractional Dunkl Laplacian $(-\Delta_k)^s$ for $s>0$ using Dunkl transform by the formula
$$ \DT((-\Delta_k)^sf)(\xi) 
=|\xi|^{2s} \DT(f)(\xi).$$
The main result of this section is the following Hardy inequality for the fractional Laplacian $(-\Delta_k)^s$. 

\begin{thm} \label{fractionalhardy}
For all $0\leq s < \frac{N+2\gamma}{2}$ and all $f\in\mathcal{S}(\RR^N)$ we have
$$ C(s)^2 \Intr_{\RR^N} \frac{|f(x)|^2}{|x|^{2s}}\diff \mu_k 
\leq \IntN |(-\Delta_k)^{s/2} f|^2 \diff\mu_k,$$
with sharp constant given by
\begin{equation} \label{C(s)}
C(s)=2^s\frac{\Gamma(\frac{1}{2}(N/2+\gamma+s))}{\Gamma(\frac{1}{2}(N/2+\gamma-s))}.
\end{equation}
\end{thm}

This will follow easily from the following Pitt's inequality.

\begin{prop} [\cite{GIT}] \label{pittsineq}
Let $0\leq s < \frac{N+2\gamma}{2}$. Then, for all $g\in \mathcal{S}(\RR^N)$, the following inequality holds 
$$C(s) \norm{|\xi|^{-s}\mathcal{D}_k(g)(\xi)}_2 \leq  \norm{|x|^sg(x)}_2 ,$$
with sharp constant $C(s)$ given by (\ref{C(s)}).
\end{prop}

\begin{proof}[Proof of Theorem \ref{fractionalhardy}]

The Hardy inequality for fractional Laplacian is essentially a rewriting of Pitt's inequality. Indeed, let $f=\DT(g)$, then by Proposition \ref{pittsineq} it follows that
$$ C(s) \norm{|x|^{-s} f}_2 \leq \norm{|\xi|^s \DT^{-1}(f)(\xi)}_2 
= \norm{|\xi|^s \DT(f)}_2.$$
Here in the last step we used the property that $\DT^{-1}(f)(\xi) = \DT(f)(-\xi)$ and a change of variables $\xi \mapsto -\xi$. But, by Plancherel's formula
$$ \norm{|\xi|^s \DT(f)(\xi)}_2 
= \norm{\DT((-\Delta_k)^{s/2} f)}_2 
= \norm{(-\Delta_k)^{s/2} f}_2.$$
Thus, we have obtained that
$$ C(s)^2 \IntN \frac{|f(x)|^2}{|x|^{2s}} \diff \mu_k \leq \IntN |(-\Delta_k)^{s/2} f|^2 \diff\mu_k,$$
as required.
\end{proof}

In the particular case $s=1$, we obtain the classical $L^2$ Hardy inequality for Dunkl operators from Theorem \ref{hardy}. To see that this inequality does indeed follow from Theorem \ref{fractionalhardy} it is enough to check that 
$$ \IntN |\nabla_k f|^2 \diff\mu_k = \IntN |(-\Delta_k)^{1/2} f|^2 \diff\mu_k.$$
By Plancherel's formula and the definition of fractional Laplacian, we obtain 
\begin{align*}
\IntN |(-\Delta_k)^{1/2} f|^2 \diff\mu_k
= \IntN |\DT((-\Delta_k)^{1/2} f)|^2 \diff\mu_k(\xi)
= \IntN |\xi|^{2} |\DT(f)(\xi)|^2 \diff\mu_k(\xi)
\\
= \IntN \DT((-\Delta_k) f)(\xi) \cdot \overline{\DT(f)} \diff\mu_k(\xi)
= -\IntN  \Delta_k f \cdot \overline{f} \diff\mu_k
= \IntN |\nabla_k f|^2 \diff\mu_k,
\end{align*}
where in the end we used integration by parts.

\section{Improved Hardy's inequalities} \label{SEC:improvedhardy}

It is a well known fact that the best constant in the classical Hardy inequality is not achieved, i.e., there is no $f\neq 0$ such that 
$$\IntN |\nabla f|^2 = \frac{(N-2)^2}{4} \IntN \frac{f^2}{|x|^2}.$$
Based on this observation, improved Hardy inequalities were proved, where the error in the classical Hardy inequality is bounded from below, usually by a term involving a suitable potential $V$. More precisely, inequalities of the following form are studied
\begin{equation} \label{BVimprovedhardy}
\int_\Omega |\nabla f|^2 \diff x - \frac{(N-2)^2}{4} \int_{\Omega} \frac{f^2}{|x|^2} \diff x 
\geq \int_\Omega |V| f^2 \diff x.
\end{equation}
The first such result was proved by Brezis and V\'azquez in \cite{BV} where it was used in the study of singular extremal solutions of a semilinear elliptic equation. In that paper the above inequality is proved for constant potential $V$ that depends on the domain $\Omega$. The proof is based on a symmetrisation argument. However, the authors noticed that even in this case the best constant is not achieved, so they posed the question of whether the improvement appearing on the right hand side of inequality (\ref{BVimprovedhardy}) is just the first term of a series. This was answered positively by Filippas and Tertikas in \cite{FT}, where such a construction can be found. Similar improved inequalities have also been found for other Hardy type inequalities, for example Hardy inequality on domains, with weights depending on the distance to the boundary. 

In this section we will prove an improved Hardy's inequality for Dunkl operators using a method similar to \cite{FT} based on spherical h-harmonics.

%%%%%%%%%%%%%%%%%%%%%%%%%%%%%%%%%%%%%%%%%%%%%%%%%%%%%%%%%%%%%%%%%%%%%%%%%%%%%%%%%%%%%%%%%%%%%%%%%
\begin{thm} \label{improvedhardy}
Let $X(t)=(1-\log t)^{-1}$. Let $\Omega\subset \RR^N$ be a bounded domain with $\delta= \displaystyle\sup_{x\in\Omega} |x|$. Then there exists a constant $C>0$ such that, for any $f\in C_0^\infty(\Omega)$ we have the inequality
\begin{align*}
\int_\Omega \left[|\nabla_k f|^2-\frac{(N+2\gamma-2)^2}{4}\frac{f^2}{|x|^2}\right] \diff\mu_k
\geq C \left(\int_\Omega |f|^q X^{1+q/2}\left(\frac{|x|}{\delta}\right) \diff\mu_k\right)^{2/q},
\end{align*}
where $q=\frac{2(N+2\gamma)}{N+2\gamma-2}$ is the Sobolev coefficient.
\end{thm}
%%%%%%%%%%%%%%%%%%%%%%%%%%%%%%%%%%%%%%%%%%%%%%%%%%%%%%%%%%%%%%%%%%%%%%%%%%%%%%%%%%%%%%%%%%%%%%%%%

We will need the following Lemma, which is itself a weighted Hardy inequality in one dimension. For a comprehensive treatment of such inequalities, see \cite{OK} and \cite{Mazya}. The following Lemma is a direct consequence of Theorem 1.3.2/3 in \cite{Mazya}.

%%%%%%%%%%%%%%%%%%%%%%%%%%%%%%%%%%%%%%%%%%%%%%%%%%%%%%%%%%%%%%%%%%%%%%%%%%%%%%%%%%%%%%%%%%%%%%%%%
\begin{lem} \label{improvedhardylemma}
Let $q\geq 2$ and $\delta>0$. Then there exists a constant $C>0$ such that the following inequality holds for all $g\in C_0^\infty(0,\delta)$
$$ \int_0^\delta t |g'(t)|^2 \diff t 
\geq C \left(\int_0^\delta \frac{|g(t)|^q}{t} X\left(\frac{t}{\delta}\right)^{1+\frac{q}{2}} \diff t\right)^{2/q}.$$
\end{lem}
%%%%%%%%%%%%%%%%%%%%%%%%%%%%%%%%%%%%%%%%%%%%%%%%%%%%%%%%%%%%%%%%%%%%%%%%%%%%%%%%%%%%%%%%%%%%%%%%%

\begin{proof} [Proof of Theorem \ref{improvedhardy}]

We note first that it is enough to prove the result in the case when $\Omega$ is the ball $B_\delta$ of radius $\delta$ centred at the origin. Indeed, for a general $\Omega$ we have $\Omega \subset B_\delta$, so the result for $B_\delta$ implies in particular the inequality for $\Omega$. 

Consider the $h$-harmonic expansion of a function $f\in L^2(\mu_k)$ 
$$ f(r\xi) 
= \Sumn_{n=0}^\infty \Sumn_{i=1}^{d(n)} f_{n,i}(r)Y^n_i(\xi).$$
The functions $f_{n,i}$, given by (\ref{hharmonicsintegral}), are defined on $\RR_+$, but with a slight abuse of notation we will also see them as radial functions on $\RR^N$ by identifying $f_{n,i}(x)=f_{n,i}(|x|)$. Using the formula (\ref{hharmonicslaplacian}) for the Dunkl Laplacian, we have
$$ \Intr_\Omega |\nabla_k f|^2 \diff\mu_k
= -\Intr_\Omega f\Delta_k f \diff\mu_k 
= -\Intr_\Omega f\left[ \frac{\partial^2 f}{\partial r^2}+\frac{N+2\gamma-1}{r}\frac{\partial f}{\partial r}+\frac{1}{r^2}\Delta_{k,0}f\right] \diff\mu_k.$$

Using the orthogonality of the h-harmonics $\{Y_i^n\}$, as well as the fact that $Y_i^n$ are eigenfunctions of the operator $\Delta_{k,0}$ with eigenvalues $\lambda_n$, it follows that 
\begin{align*}
&\int_\Omega |\nabla_k f|^2  \diff\mu_k
\\
&\qquad
= - \Sumn_{n=0}^\infty \Sumn_{i=1}^{d(n)} \int_0^\delta \left[f_{n,i}(r)f_{n,i}''(r)+\frac{N+2\gamma-1}{r}f_{n,i}(r)f_{n,i}'(r)
+\lambda_n\frac{1}{r^2}f^2_{n,i}(r)\right] r^{N+2\gamma-1}\diff r 
\\
&\qquad
= - \Sumn_{n=0}^\infty \Sumn_{i=1}^{d(n)} \int_0^\delta \left[f_{n,i}(x)\Delta_k f_{n,i}(x)+\lambda_n\frac{f_{n,i}^2(x)}{|x|^2}\right]\diff x 
\\
&\qquad
= \Sumn_{n=0}^\infty \Sumn_{i=1}^{d(n)} \int_0^\delta \left[|\nabla_k f_{n,i}(x)|^2-\lambda_n\frac{f_{n,i}^2(x)}{|x|^2}\right]\diff x.
\end{align*}

Let $\Lambda := \left(\frac{N+2\gamma-2}{2}\right)^2$ denote the Hardy inequality constant. From the above, we have
\begin{align*}
\int_\Omega \left[|\nabla_k f|^2-\Lambda\frac{f^2}{|x|^2}\right] \diff\mu_k = \Sumn_{n=0}^\infty \Sumn_{i=1}^{d(n)} I_{n,i},
\end{align*}
where 
$$ I_{n,i}:= \int_\Omega \left[|\nabla_k f_{n,i}|^2-(\Lambda+\lambda_n)\frac{f_{n,i}^2(x)}{|x|^2}\right]\diff \mu_k.$$

When $n>0$, it can be checked by rearranging the terms that Hardy's inequality implies the following
$$ I_{n,i}\geq \frac{\lambda_n}{\lambda_n-\Lambda} \int_\Omega \left[ |\nabla_k f_{n,i}|^2 -\lambda_n \frac{f_{n,i}^2}{|x|^2}\right] \diff\mu_k.$$
Thus, we have that
\begin{align} \label{improvedhardyeqn1}
\Sumn_{n=1}^\infty \Sumn_{i=1}^{d(n)} I_{n,i}
\geq C_1 \Sumn_{n=1}^\infty \Sumn_{i=1}^{d(n)} \int_\Omega \left[ |\nabla_k f_{n,i}|^2 -\lambda_n \frac{f_{n,i}^2}{|x|^2}\right] \diff\mu_k
= C_1 \int_\Omega |\nabla_k(f-f_{0,1})|^2 \diff\mu_k,
\end{align}
where $C_1=\displaystyle\min_{n\geq 1} \frac{\lambda_n}{\lambda_n-\Lambda}>0$. Using the Sobolev inequality, we have that
\begin{align} \label{ineq1}
\int_\Omega |\nabla_k(f-f_{0,1})|^2 \diff\mu_k 
&\geq C_2 \left(\int_\Omega |f-f_{0,1}|^q \diff\mu_k\right)^{2/q} \nonumber \\
&\geq C_2 \left(\int_\Omega |f-f_{0,1}|^q X^{1+q/2}\left(\frac{|x|}{\delta}\right)\diff\mu_k\right)^{2/q},
\end{align}
for a constant $C_2>0$. Here, in the second inequality, we used the fact that $X$ is bounded above by $1$.

When $n=0$, we have that $d(0)=1$ and $\lambda_0=0$, so 
$$ I_{0,1}=\int_\Omega \left[ |\nabla_k f_{0,1}|^2-\Lambda \frac{f_{0,1}^2}{|x|^2}\right] \diff\mu_k. $$
Let $u(r)=r^{(N+2\gamma-2)/2} f_{0,1}$ so after an easy computation we find that
$$ I_{0,1} = \int_\Omega |x|^{-(N+2\gamma-1)} \left[ - (N+2\gamma-2) u(|x|)u'(|x|) + |x| u'(|x|)^2 \right] \diff\mu_k(x).$$
Using polar coordinates we then have 
$$ I_{0,1}= p(B_1) \int_0^\delta [- (N+2\gamma-2) u(r)u'(r)+r u'(r)^2] \diff r=p(B_1) \int_0^\delta r u'(r)^2 \diff r,$$
where $p(B_1) = \int_{\mathbb{S}^{N-1}} w_k(\theta) \diff \sigma(\theta)$ and in the last equality we used the fact that $u(0)=u(\delta)=0$. Applying Lemma \ref{improvedhardylemma}, this implies
\begin{align} \label{ineq2}
I_{0,1}
&\geq p(B_1) C_3 \left(\int_0^\delta \frac{|u|^q}{r} X^{1+q/2}\left(\frac{r}{\delta}\right) \diff r\right)^{2/q} \nonumber \\
&=p(B_1)^{1-2/q}C_3 \left(\int_\Omega|f_{0,1}|^q X^{1+q/2}\left(\frac{|x|}{\delta}\right) \diff \mu_k\right)^{2/q}.
\end{align}

Finally, from (\ref{improvedhardyeqn1}) we obtain
\begin{align*}
\Intr \left[|\nabla_k f|^2-\Lambda\frac{f^2}{|x|^2}\right] \diff\mu_k
&\geq I_{0,1}+C_1 \Intr |\nabla_k(f-f_{0,1})|^2 \diff\mu_k \\
&\geq C \left(\Intr |f|^q X^{1+q/2}(|x|) \diff\mu_k\right)^{2/q},
\end{align*}
for a constant $C>0$, where in the last line we used (\ref{ineq1}), (\ref{ineq2}), and the triangle inequality in the space $L^q(\mu_k)$.
\end{proof}

The following Corollary, which is a Dunkl equivalent of the original result of Brezis and Vazquez, is very important because it establishes a Poincar\'e inequality for Dunkl operators.

%%%%%%%%%%%%%%%%%%%%%%%%%%%%%%%%%%%%%%%%%%%%%%%%%%%%%%%%%%%%%%%%%%%%%%%%%%%%%%%%%%%%%%%%%%%%%%%%%
\begin{cor}
Let $\Omega \subset \RR^N$ be a bounded domain. Then there exists a constant $C(\Omega)>0$ such that for any $f\in C^\infty_0(\Omega)$ we have the inequality
$$ \int_\Omega |\nabla_k f|^2 \diff\mu_k \geq \frac{(N+2\gamma-2)^2}{4} \int_\Omega \frac{f^2}{|x|^2} \diff\mu_k + C(\Omega) \int_\Omega f^2 \diff\mu_k.$$
\end{cor}
%%%%%%%%%%%%%%%%%%%%%%%%%%%%%%%%%%%%%%%%%%%%%%%%%%%%%%%%%%%%%%%%%%%%%%%%%%%%%%%%%%%%%%%%%%%%%%%%%

\begin{proof}
This follows from the previous Theorem and H\"older's inequality applied to the function $f^2 X^{1+\frac{2}{q}} \in L^{q/2}$ and $X^{-1-\frac{2}{q}} \in L^{q/(q-2)}$.
\end{proof}

%%%%%%%%%%%%%%%%%%%%%%%%%%%%%%%%%%%%%%%%%%%%%%%%%%%%%%%%%%%%%%%%%%%%%%%%%%%%%%%%%%%%%%%%%%%%%%%%%
\begin{cor} [Poincar\'e inequality]
Let $\Omega \subset \RR^N$ be a bounded domain. Then there exists a constant $C(\Omega)>0$ such that for any $f\in C^\infty_0(\Omega)$ we have the inequality
$$ \int_\Omega |\nabla_k f|^2 \diff\mu_k \geq C(\Omega) \int_\Omega f^2 \diff\mu_k.$$
\end{cor}
%%%%%%%%%%%%%%%%%%%%%%%%%%%%%%%%%%%%%%%%%%%%%%%%%%%%%%%%%%%%%%%%%%%%%%%%%%%%%%%%%%%%%%%%%%%%%%%%%

\section{Other Hardy-type inequalities} \label{SEC:otherhardy}

In this section we present two results closely related to the Hardy inequality: the Rellich inequality and a Caffarelli-Kohn-Nirenberg inequality. 

\subsection{The Rellich Inequality}

This classical inequality, first proved by Rellich in \cite{Rellich0} (see also \cite{Rellich}), states that for all $f\in C^\infty_c(\RR^N \setminus \{0\})$ we have
$$ \IntN |\Delta f|^2 \diff x \geq \frac{N^2 (N-4)^2}{16} \IntN \frac{f^2}{|x|^4} \diff x,$$
where the constant is sharp. 

Here we prove the Dunkl analogue of this inequality. Our proof below uses the method of spherical h-harmonics already employed above to obtain an improved Hardy inequality, and it is similar in style to the original proof of Rellich.  

%%%%%%%%%%%%%%%%%%%%%%%%%%%%%%%%%%%%%%%%%%%%%%%%%%%%%%%%%%%%%%%%%%%%%%%%%%%%%%%%%%%%%%%%%%%%%%%%%
\begin{thm}[Rellich inequality] \label{rellich2}
Suppose that $N+2\gamma \neq 2$. Then, for any $f\in C_0^\infty(\RR^N \setminus \{ 0\})$, we have the inequality
$$ \IntN |\Delta_k f|^2 \diff\mu_k 
\geq \frac{(N+2\gamma)^2(N+2\gamma-4)^2}{16}\IntN \frac{f^2}{|x|^4} \diff\mu_k.$$ 
The constant in this inequality is sharp.
\end{thm}
%%%%%%%%%%%%%%%%%%%%%%%%%%%%%%%%%%%%%%%%%%%%%%%%%%%%%%%%%%%%%%%%%%%%%%%%%%%%%%%%%%%%%%%%%%%%%%%%%

\begin{proof}
Consider the expansion of $f$ in terms of spherical h-harmonics
$$ f(x)=\Sumn_{n=0}^\infty\Sumn_{i=1}^{d(n)} f_{n,i}(r)Y_i^n(\xi).$$
We then have
\begin{align*}
\Delta_k f(x) = \Sumn_{n=0}^\infty\Sumn_{i=1}^{d(n)} \left[f''_{n,i}(r)+\frac{N+2\gamma-1}{r}f'_{n,i}(r)+\frac{\lambda_n}{r^2}f_{n,i}(r)\right] Y_i^n(\xi).
\end{align*}
In order to simplify the computations below, we introduce the notation $\overline{N}:=N+2\gamma$. Thus, from the orthogonality properties of $\{ Y_{n,i}\}$ we have
\begin{align*}
\IntN (\Delta_k f)^2 \diff\mu_k 
&= \Sumn_{n=0}^\infty\Sumn_{i=1}^{d(n)} \int_0^\infty \left[f''_{n,i}(r)+\frac{\overline{N}-1}{r}f'_{n,i}(r)+\frac{\lambda_n}{r^2}f_{n,i}(r)\right]^2 r^{\overline{N}-1}\diff r.
\end{align*}
Expanding the brackets and computing the terms containing products of mixed derivatives using integration by parts (recall that $f_{n,i}$ has compact support away from 0), we have
\begin{align*}
\IntN (\Delta_k f)^2 \diff\mu_k
&= \Sumn_{n=0}^\infty\Sumn_{i=1}^{d(n)} \int_0^\infty \left[(f''_{n,i})^2r^{\overline{N}-1}
+(\overline{N}-2\lambda_n-1)(f'_{n,i})^2r^{\overline{N}-3} \right.
\\
& \qquad \qquad \qquad \qquad 
\left. +\lambda_n(\lambda_n-2(\overline{N}-4))f_{n,i}^2r^{\overline{N}-5}\right]\diff r.
\end{align*}

Fix now some $n=0,1,\ldots$, and some $i=1,\ldots, d(n)$, and define
$$ g_{n,i}=r^{(\overline{N}-4)/2}f_{n,i}.$$
To simplify notation, let $u=f_{n,i}$ and $v=g_{n,i}$; let also $c=\overline{N}-2\lambda_n-1$. We then have
\begin{equation} \label{rellich2relations}
\begin{aligned}
u'(r)&=\frac{4-\overline{N}}{2}r^{(2-\overline{N})/2}v+r^{(4-\overline{N})/2}v'
\\
u''(r)&=\frac{(4-\overline{N})(2-\overline{N})}{4}r^{-\overline{N}/2}v+(4-\overline{N})r^{(2-\overline{N})/2}v'+r^{(4-\overline{N})/2}v''.
\end{aligned}
\end{equation}

Thus, using these relations and integrating by parts the terms involving mixed products of derivatives of $v$, we have 
\begin{align*}
\Intr_0^\infty &\left[(f''_{n,i})^2r^{\overline{N}-1}
+(\overline{N}-2\lambda_n-1)(f'_{n,i})^2r^{\overline{N}-3}\right] \diff r
 = \Intr_0^\infty \left[(u'')^2r^{\overline{N}-1}+c(u')^2r^{\overline{N}-3}\right]^2 \diff r
\\
&= \Intr_0^\infty \left[(v'')^2 r^3 
+ ((4-\overline{N})^2+c)(v')^2 r
+\frac{(4-\overline{N})^2}{4}\left(\frac{(2-\overline{N})^2}{4}+c\right)v^2r^{-1}\right.
\\
&\quad\quad\qquad + 2\frac{4-\overline{N}}{2}\left(\frac{(4-\overline{N})(2-\overline{N})}{2}+c\right)vv'
\\
&\qquad\qquad+\left.2\frac{(4-\overline{N})(2-\overline{N})}{4}vv''r+2(4-\overline{N})v'v'' r^2\right] \diff r
\\
&=\Intr \left[(v'')^2r^3
+\left(c+\frac{(4-\overline{N})(2-\overline{N})}{2}\right)(v')^2r
+\frac{(4-\overline{N})^2}{4}\left(\frac{(2-\overline{N})^2}{4}+c\right)v^2r^{-1}\right] \diff r
\\
&=\frac{\overline{N}^2(4-\overline{N})^2}{16}\Intr_0^\infty v^2r^{-1}\diff r
\\
&\qquad + \Intr_0^\infty \left[(v'')^2r^3
+\left(\frac{1}{2}(\overline{N}-2)^2+1-2\lambda_n\right)(v')^2r-\lambda_n\frac{(4-\overline{N})^2}{2}v^2r^{-1}\right] \diff r.
\end{align*}
It then follows that
\begin{align*}
\Intr_0^\infty &((f''_{n,i})^2r^{\overline{N}-1}
+(\overline{N}-2\lambda_n-1)(f'_{n,i})^2r^{\overline{N}-3}
+\lambda_n(\lambda_n-2(\overline{N}-4))f_{n,i}^2r^{\overline{N}-5}) \diff r
\\
&\geq \frac{\overline{N}^2(4-\overline{N})^2}{16}\Intr_0^\infty v^2r^{-1}\diff r 
+ \Intr_0^\infty \left[(v'')^2r^3
+A (v')^2r + Bv^2r^{-1}\right] \diff r,
\end{align*}
where
\begin{align*}
A&=\frac{1}{2}(\overline{N}-2)^2+1-2\lambda_n =  2n(n+\overline{N}-2) + \frac{1}{2}(\overline{N}-2)^2+1
\\
B&= \lambda_n(\lambda_n-2(\overline{N}-4))-\lambda_n\frac{(4-\overline{N})^2}{2}
= n(n+\overline{N}-2) \left[ n(n+\overline{N}-2) + \frac{\overline{N}}{2}(\overline{N}-4) \right].
\end{align*}
It is then clear that $A \geq 0$ for all $n=0,1,\ldots$ without any restrictions on $\overline{N}$, whilst $B \geq 0$ for all $n$ as long as $\overline{N}\neq 2$ (which is why we made this assumption). Finally, we have obtained
\begin{align*}
\Intr_0^\infty ((f''_{n,i})^2r^{\overline{N}-1}
+&(\overline{N}-2\lambda_n-1)(f'_{n,i})^2r^{\overline{N}-3}
+\lambda_n(\lambda_n-2(\overline{N}-4))f_{n,i}^2r^{\overline{N}-5}) \diff r
\\
&\geq \frac{\overline{N}^2(4-\overline{N})^2}{16}\Intr_0^\infty v^2r^{-1}\diff r 
= \frac{\overline{N}^2(4-\overline{N})^2}{16} \Intr_0^\infty f^2_{n,i} r^{\overline{N}-5} \diff r.
\end{align*} 

Adding these up for all $n=0, 1, \ldots$ and $i=1,\ldots, d(n)$, from the above, and reconstructing $f$ back from its spherical h-harmonics components, we have obtained that 
$$ \IntN (\Delta_k f)^2 \diff\mu_k 
\geq \frac{(N+2\gamma)^2(N+2\gamma-4)^2}{16} \IntN \frac{f^2}{|x|^4} \diff\mu_k,$$
as required. 

To check that the constant is sharp we can use a similar example as in the classical case. More precisely, for $n=3,4 \ldots$ let $f_n(x)=|x|^{2-\frac{N+2\gamma}{2}} h_n(|x|)$, where 
$$h_n : [0,\infty) \to [0,1] \quad \text{ is such that } \quad h_n(r)=
\begin{cases}
0 & \text{ if } r \leq 1 \\
1 & \text{ if } 2\leq r \leq n \\
0 & \text{ if } r\geq 2n,
\end{cases}
$$
with derivatives satisfying 
$$|h_n'| \leq \frac{c_1}{n} \quad \text{ and }  \quad |h_n''| \leq \frac{c_2}{n^2}$$
for some constants $c_1,c_2>0$ (an explicit such $h_n$ can be found in \cite{Rellich}). Then we can compute
$$ \IntN |\nabla_k f_n|^2 \diff\mu_k = C_1 + p(B_1)\frac{(N+2\gamma)^2(N+2\gamma-4)^2}{16} \int_2^{2n} \frac{1}{r} \diff r$$
and 
$$ \IntN \frac{f_n^2}{|x|^4} \diff\mu_k = C_2 + p(B_1) \int_2^{2n} \frac{1}{r} \diff r,$$
where $C_1$ and $C_2$ can be bounded by constants that do not depend on $n$. Thus
$$ \displaystyle\lim_{n\to\infty} \frac{\IntN |\nabla_k f_n|^2 \diff\mu_k}{\IntN \frac{f_n^2}{|x|^4} \diff\mu_k} = \frac{(N+2\gamma)^2(N+2\gamma-4)^2}{16}.$$
\end{proof}

\subsection{The Caffarelli-Kohn-Nirenberg Inequality}

The Caffarelli-Kohn-Nirenberg inequality \cite{CKN} is
$$ \left(\IntN \frac{|f|^p}{|x|^{pb}} \diff x \right)^{1/p} \leq C \left( \IntN \frac{|\nabla f|^r}{|x|^{ra}} \diff x \right)^{\theta/r} \left( \IntN \frac{|f|^q}{|x|^{qc}} \diff x \right)^{(1-\theta)/q}.$$ 
In \cite{CKN} necessary and sufficient conditions on the parameters $p,q,r,a,b,c,\theta$ are given for which the above inequality holds for all $f\in C^\infty_c(\RR^N)$. We prove here the Dunkl analogue of a particular case of this inequality, corresponding to the values $r=2$ and $\theta=1$. In this special case the inequality was known before the work of Caffarelli, Kohn and Nirenberg, see for example \cite{Ilin}, and it it sometimes known as the Hardy-Sobolev inequality as it generalises both these results (which correspond to the values $a=0$, $b=0$, and $a=0$, $b=1$, respectively). Best constants are also known, see \cite[Corollary 4.8]{Mazya}. Our proof below is inspired by the method of \cite{GR}.

%%%%%%%%%%%%%%%%%%%%%%%%%%%%%%%%%%%%%%%%%%%%%%%%%%%%%%%%%%%%%%%%%%%%%%%%%%%%%%%%%%%%%%%%%%%%%%%%%
\begin{thm} \label{caffarelli-kohn-nirenberg}
Let $a\leq b \leq a+1$ and $p=\frac{2(N+2\gamma)}{N+2\gamma-2+2(b-a)}$, and suppose that $a<\frac{N+2\gamma-2}{2}$. Then, for any $f\in C_c^\infty(\RR^N)$ we have the inequality
$$ \int_{\RR^N} \frac{|\nabla_k f|^2}{|x|^{2a}} \diff\mu_k \geq C_{a,b} \left(\int_{\RR^N} \frac{|f|^p}{|x|^{pb}} \diff\mu_k\right)^{2/p},$$
where $C_{a,b}>0$ is a constant. 
\end{thm}
%%%%%%%%%%%%%%%%%%%%%%%%%%%%%%%%%%%%%%%%%%%%%%%%%%%%%%%%%%%%%%%%%%%%%%%%%%%%%%%%%%%%%%%%%%%%%%%%%

\begin{proof}
The strategy of the proof is to establish the inequality in the end cases $b=a+1$ and $b=a$ separately, and then to interpolate between these two cases to obtain the result in full generality.

\textbf{Step 1.} Suppose $b=a+1$, so $p=2$. We begin by considering the function $u=\frac{f}{|x|^{a}}$, so 
\begin{equation} \label{leibnizckn}
T_i u = \frac{T_i f}{|x|^{a}} - a \frac{x_i}{|x|^{a+2}} f.
\end{equation}
Then we have
\begin{equation} \label{ckn1}
\begin{aligned}
\int_{\RR^N} \frac{|\nabla_k f|^2}{|x|^{2a}} \diff\mu_k
&= \int_{\RR^N} \left| \nabla_k u + a \frac{x}{|x|^{a+2}} f \right|^2 \diff\mu_k 
\\
&= \int_{\RR^N} |\nabla_k u|^2 \diff\mu_k 
	+ a^2 \int_{\RR^N} \frac{f^2}{|x|^{2a+2}} \diff\mu_k 
	+ 2a \int_{\RR^N} \frac{f}{|x|^{a+2}} \langle x, \nabla_k u \rangle \diff\mu_k.
\end{aligned}
\end{equation}

Let $\epsilon > 1$. Applying the inequality $2xy \geq -\epsilon x^2 - \frac{1}{\epsilon} y^2$, we can estimate the last term on the right hand side of the previous equality 
$$ 2a \int_{\RR^N} \frac{f}{|x|^{a+2}} \langle x, \nabla_k u \rangle\diff\mu_k 
\geq -\epsilon a^2 \int_{\RR^N} \frac{f^2}{|x|^{2a+2}} \diff\mu_k - \frac{1}{\epsilon} \int_{\RR^N} |\nabla_k u|^2 \diff\mu_k.$$
Plugging this in (\ref{ckn1}), we have obtained 
$$ \int_{\RR^N} \frac{|\nabla_k f|^2}{|x|^{2a}} \diff\mu_k
\geq a^2(1-\epsilon) \int_{\RR^N} \frac{f^2}{|x|^{2a+2}} \diff\mu_k
+ (1-\frac{1}{\epsilon}) \int_{\RR^N} |\nabla_k u|^2 \diff\mu_k.$$
Applying Hardy's inequality to the last term on the right hand side of the above inequality, we have
$$ \int_{\RR^N} \frac{|\nabla_k f|^2}{|x|^{2a}} \diff\mu_k
\geq \left(a^2(1-\epsilon)+(1-\frac{1}{\epsilon})\frac{(N+2\gamma-2)^2}{4}\right) \int_{\RR^N} \frac{f^2}{|x|^{2a+2}} \diff\mu_k.$$
This holds for all $\epsilon>1$ and since $a<\frac{N+2\gamma-2}{2}$, we can choose for example $\epsilon=\frac{N+2\gamma-2}{2a}$ to obtain a positive constant.

\textbf{Step 2.} Suppose now that $a=b$. In this case $p=q:=\frac{2(N+2\gamma)}{N+2\gamma-2}$, the Sobolev coefficient. Using the Sobolev inequality we have
$$ \left(\int_{\RR^N} \frac{|f|^p}{|x|^{pa}} \diff\mu_k\right)^{2/p} 
\leq C \int_{\RR^N} \left|\nabla_k \left( \frac{f}{|x|^a} \right)\right|^2 \diff\mu_k. $$
Using (\ref{leibnizckn}), we obtain
$$ \left|\nabla_k \left( \frac{f}{|x|^a} \right)\right|^2
= \left| \frac{\nabla_k f}{|x|^a} - a \frac{x}{|x|^{a+2}}f \right|^2 \leq 2\frac{|\nabla_k f|^2}{|x|^{2a}} + 2a^2 \frac{f^2}{|x|^{2a+2}}.$$

Thus, from the last two relations it follows that
\begin{align*}
\left(\int_{\RR^N} \frac{|f|^p}{|x|^{pa}} \diff\mu_k\right)^{2/p} 
&\leq 2Ca^2 \int_{\RR^N} \frac{|f|^2}{|x|^{2a+2}} \diff\mu_k 
+ 2C \int_{\RR^N} \frac{|\nabla_k f|^2}{|x|^{2a}} \diff\mu_k
\\
&\leq 2C (a^2 C_{a,a+1}^{-1} +1) \int_{\RR^N} \frac{|\nabla_k f|^2}{|x|^{2a}} \diff\mu_k,
\end{align*}
where we used the previous step.

\textbf{Step 3.} We now look at the case $a<b<a+1$. As above, let $q=\frac{2(N+2\gamma)}{N+2\gamma-2}$ be the Sobolev coefficient. We have $2<p<q$, so there exists $\theta\in(0,1)$ such that
$$ p=2\theta+q(1-\theta),$$
so 
$$ b= a + \frac{\theta (N+2\gamma-2)}{N+2\gamma-2\theta},$$
and also
$$pb=2(a+1)\theta + qa(1-\theta).$$
Then, using H\"older's inequality, we obtain
\begin{align*}
\int_{\RR^N} \frac{|f|^p}{|x|^{pb}} \diff\mu_k 
&=\int_{\RR^N} \frac{|f|^{2\theta+q(1-\theta)}}{|x|^{2(a+1)\theta + qa(1-\theta)}} \diff\mu_k 
\\
&\leq \left(\int_{\RR^N} \frac{|f|^2}{|x|^{2(a+1)}} \diff\mu_k\right)^\theta 
\left( \int_{\RR^N} \frac{|f|^q}{|x|^{qa}} \diff\mu_k \right)^{1-\theta}.
\end{align*}
Using the two steps above, this implies
$$ \int_{\RR^N} \frac{|f|^p}{|x|^{pb}} \diff\mu_k
\leq C_{a,a+1}^{-\theta}C_{a,a}^{-q(1-\theta)/2} 
\left(\int_{\RR^N} \frac{|\nabla_k f|^2}{|x|^{2a}} \diff\mu_k\right)^{p/2},$$
as required. This completes the proof.
\end{proof}

\begin{remark}
One could prove a more general Caffarelli-Kohn-Nirenberg inequality of the form
$$ \left( \int_{\RR^N} \frac{|f|^p}{|x|^{pb}} \diff\mu_k \right)^{1/p} 
\leq C \left( \int_{\RR^N} \frac{|\nabla_k f|^2}{|x|^{2a}} \diff\mu_k \right)^{\theta/2} \left( \int_{\RR^N} \frac{|f|^q}{|x|^{qc}} \diff\mu_k \right)^{(1-\theta)/q},$$
which holds for all $f\in C_c^\infty (\RR^N\setminus \{0\})$, subject to the assumption that
$$\frac{1}{p}-\frac{b}{N+2\gamma}=\theta \left( \frac{1}{2} -\frac{a+1}{N+2\gamma}\right) + (1-\theta) \left( \frac{1}{q} - \frac{c}{N+2\gamma}\right),$$
where $b=(1-\theta)c+\theta d$, for parameters $p,q,a,c,d,\theta \in \RR$ such that $p>0$, $q\geq 1$, $\theta\in [0,1]$, and such that all integrals above are finite. This could be achieved by interpolating using H\"older's inequality between the case $\theta=0$ (which is trivial as $p=q$, $b=c$, and both sides reduce to $\norm{|x|^{-b}f}_p$), and the case $\theta=1$, which was done in the previous Theorem. However, this only works for a more restricted and rather complicated range of $\theta$ depending on $p$ and $q$.
\end{remark}

\section{Many-particle Hardy inequalities}

In this section we exploit the connection between Dunkl operators and generalised Calogero-Moser-Sutherland models to obtain an application to many-particle Hardy inequalities in dimension $d=1$. The generalised CMS Hamiltonian is defined for any root system $R$ as
\begin{equation} \label{Fkexpression} \mathcal{F}_k g(x):= \Delta g(x) - 2 \Suma \frac{k_\alpha}{\alx^2} (k_\alpha g(x) - g(\sigma_\alpha x)).
\end{equation}
It is known (see \cite{Rosler2000}) that this is connected to the Dunkl Laplacian via the relation
\begin{equation} \label{Dunkl-CMS} 
w_k^{-1/2} \mathcal{F}_k w_k^{1/2} = \Delta_k.
\end{equation}
This follows simply from \eqref{Dunkllaplacian} and the identity (see \cite{Dunkl1989})
$$ \sum_{\substack{\alpha,\beta \in R_+ \\ \alpha\neq \beta}} k_\alpha k_\beta \frac{\langle\alpha, \beta \rangle}{\alx \langle \beta, x \rangle} = 0 \qquad \forall x\in\RR^N.$$

The main result of this section is the following.

\begin{thm} \label{generalisedmultiHardy}
Let $\Omega = \RR^N \setminus \bigcup_{\alpha\in R_+}\left\{x\in \RR^N : \alx =0\right\}$. Then, for any $f\in C_c^\infty (\Omega; \CC)$ and for any $k_\alpha \geq 0$, we have the inequality
\begin{equation} \label{manyparticleHardy}
\IntN |\nabla f|^2 \diff x 
\geq 2 \IntN |f|^2 \Suma \frac{k_\alpha-k_\alpha^2}{\alx^2} \diff x + \frac{(N+2\gamma-2)^2}{4} \IntN \frac{|f|^2}{|x|^2} \diff x.
\end{equation}
In particular, by taking $k_\alpha=\frac{1}{2}$ for all $\alpha\in R_+$, we have 
\begin{equation} \label{manyparticleHardyoptimal}
\IntN |\nabla f|^2 \diff x 
\geq \frac{1}{2} \IntN |f|^2 \Suma \frac{1}{\alx^2} \diff x + \frac{(N+|R_+|-2)^2}{4} \IntN \frac{|f|^2}{|x|^2} \diff x.
\end{equation}
\end{thm}

\begin{proof}
Let $g=w_k^{-1/2} f \in C_c^\infty(\Omega; \CC)$. Using \eqref{Dunkl-CMS}, we deduce from the Hardy inequality of Theorem \ref{hardy} that
\begin{align} \label{multiparticleHardystep1}
- \IntN \overline{w_k^{1/2}g} \mathcal{F}_k(w_k^{1/2}g)  \diff x 
&= - \IntN \overline{g} \Delta_k g  \diff\mu_k  \nonumber
\\ 
&= \IntN |\nabla_k g|^2 \diff\mu_k 
\geq \frac{(N+2\gamma-2)^2}{4} \IntN \frac{|g|^2}{|x|^2} \diff\mu_k.
\end{align}
Using \eqref{Fkexpression}, we can compute the right hand side of this inequality as follows:
\begin{equation} \label{multiparticleHardystep2}
\begin{aligned} 
- \IntN \overline{w_k^{1/2}g} \mathcal{F}_k(w_k^{1/2}g)  \diff x 
&= \IntN |\nabla f|^2 \diff x + 2 \Suma k_\alpha^2 \IntN \frac{|f|^2}{\alx^2} \diff x 
\\
&\qquad - 2 \Suma k_\alpha \IntN \frac{\overline{f(x)} f(\sigma_\alpha x)}{\alx^2} \diff x.
\end{aligned}
\end{equation}

Assume first that the support of $f$ is contained in a Weyl chamber $H$ and consider the function $\tilde{f}$ defined by 
$$\tilde{f}(\sigma_\alpha x)= f(x) \qquad \forall x\in H,\; \forall\alpha\in R.$$
In other words, $\tilde{f}$ is a $G$-invariant function obtained by copying the function $f$ (defined on $H$) identically on each Weyl chamber. From \eqref{multiparticleHardystep1} and \eqref{multiparticleHardystep2} applied to the $G$-invariant function $\tilde{f}$, we have
\begin{equation*}
\IntN |\nabla \tilde{f}|^2 \diff x 
\geq 2 \Suma (k_\alpha - k_\alpha^2) \IntN \frac{|\tilde{f}|^2}{\alx^2} \diff x
+ \frac{(N+2\gamma-2)^2}{4} \IntN \frac{|\tilde{f}|^2}{|x|^2} \diff x.
\end{equation*}
By $G$-invariance and recalling that $\tilde{f}=f$ on $H$, this implies that
\begin{equation} \label{multiHardy_Weylchamber}
\int_H |\nabla f|^2 \diff x 
\geq 2 \Suma (k_\alpha - k_\alpha^2) \int_H \frac{|f|^2}{\alx^2} \diff x
+ \frac{(N+2\gamma-2)^2}{4} \int_H \frac{|f|^2}{|x|^2} \diff x,
\end{equation}
which is what we wanted to prove.

For general $f\in C^\infty_c (\Omega; \CC)$, let $f=\sum_H f_H$ where the sum goes over all Weyl chambers $H$ and $f_H=f$ on $H$ and vanishes elsewhere. Thus $f_H \in C_c^\infty(H; \CC)$ and so we obtain \eqref{multiHardy_Weylchamber} for each $H$. Adding up all the resulting inequalities, we obtain the desired inequality \eqref{manyparticleHardy}.

The constant in front of the first term on the right hand side of \eqref{manyparticleHardy} is maximised for $k_\alpha=\frac{1}{2}$, and in this case we have
$$\IntN |\nabla f|^2 \diff x 
\geq \frac{1}{2} \IntN |f|^2 \Suma \frac{1}{\alx^2} \diff x + \frac{(N+|R_+|-2)^2}{4} \IntN \frac{|f|^2}{|x|^2} \diff x,$$
as required.
\end{proof}

By taking the root system $A_{N-1}$  with $R_+=\{ e_i - e_j : 1 \leq i < j \leq N\}$, we obtain an improvement of the well-known Calogero-Moser-Sutherland many-particle Hardy inequality. Indeed, in this case all the roots belong to the same orbit of the reflection group $G=S_N$, so the multiplicity function is a constant $k_\alpha = k$ for all $\alpha\in R_+$, and $\gamma = k \frac{N(N-1)}{2}$.

\begin{cor} \label{multi-hardy-cor1}
Let $\Omega = \RR^N \setminus \bigcup_{1\leq i <j \leq N}\left\{ x : x_i = x_j\right\}$. For any $f\in C_c^\infty (\Omega; \CC)$, and for any $k \geq 0$, we have the inequality
\begin{align*}
\IntN |\nabla f|^2 \diff x 
&\geq 2(k-k^2) \IntN |f|^2 \sum_{1 \leq i < j \leq N} \frac{1}{(x_i-x_j)^2} \diff x 
\\
&\qquad\qquad
+ \frac{(N+kN(N-1)-2)^2}{4} \IntN \frac{|f|^2}{|x|^2} \diff x.
\end{align*}
In particular, by taking $k=\frac{1}{2}$, we have
$$ \IntN |\nabla f|^2 \diff x 
\geq \frac{1}{2} \IntN |f|^2 \sum_{1 \leq i < j \leq N} \frac{1}{(x_i-x_j)^2} \diff x + \frac{(N^2+N-4)^2}{16} \IntN \frac{|f|^2}{|x|^2} \diff x.$$
\end{cor}

We also record here the Hardy inequality corresponding to a root system of type $B_N$. In this case $R_+ = \{ \sqrt{2} e_i : 1 \leq i \leq N \} \bigcup \{ e_i \pm e_j : 1 \leq i <j \leq N\}$ and the multiplicity function reduces to two constants $k_1, k_2 \geq 0$. 

\begin{cor} \label{multi-hardy-cor2}
Let $\Omega = \RR^N \setminus \bigcup_{1\leq i \leq N} \left\{x : x_i=0 \right\} \setminus \bigcup_{1\leq i < j \leq N} \left\{x: x_i=\pm x_j\right\}$. For any $f \in C_c^\infty (\Omega; \CC)$ and for any $k_1, k_2 \geq 0$ we have the inequality
\begin{align*}
\IntN |\nabla f|^2 \diff x 
&\geq (k_1-k_1^2) \IntN |f|^2 \sum_{i=1}^N \frac{1}{x_i^2} \diff x 
\\
&\qquad 
+ 2(k_2-k_2^2) \IntN |f|^2 \sum_{1 \leq i <j \leq N} \left[ \frac{1}{(x_i-x_j)^2} +\frac{1}{(x_i+x_j)^2} \right] \diff x
\\
&\qquad
+ \frac{(N+2k_1N+k_2N(N-1) -2)^2}{4} \IntN \frac{|f|^2}{|x|^2} \diff x.
\end{align*}
\end{cor}

\begin{remark}
By taking $k_1=\frac{1}{2}$ and $k_2=0$ in Corollary \ref{multi-hardy-cor2}, we recover the inequality
$$ \IntN |\nabla f|^2 \diff x \geq \frac{1}{4} \IntN |f|^2 \sum_{i=1}^N \frac{1}{x_i^2} \diff x + (N-1)^2 \IntN \frac{|f|^2}{|x|^2} \diff x,$$
which was proved in \cite{Tidblom}, where it was also shown that the constants $\frac{1}{4}$ and $(N-1)^2$ are optimal. Similarly, Corollary \ref{multi-hardy-cor1} extends to general $k\geq 0$ the results of \cite{Lundholm} where the following inequality was proved
$$ \IntN |\nabla f|^2 \diff x \geq \frac{1}{2} \IntN |f|^2 \sum_{1 \leq i < j \leq N} \frac{1}{(x_i-x_j)^2} \diff x.$$
\end{remark}

\noindent \textbf{Acknowledgements.} The author wishes to thank Sergey Tikhonov and Hatem Mejjaoli for pointing out missing references. Financial support from EPSRC is also gratefully acknowledged.

\bibliographystyle{plain}
\bibliography{ref}

\begin{thebibliography}{10}

\bibitem{Anker}
J.-Ph. Anker.
\newblock An introduction to {D}unkl theory and its analytic aspects.
\newblock In {\em Analytic, algebraic and geometric aspects of differential
  equations}, Trends Math., pages 3--58. Birkh\"{a}user/Springer, Cham, 2017.

\bibitem{BEL}
A.~A. Balinsky, W.~D. Evans, and R.~T. Lewis.
\newblock {\em The analysis and geometry of {H}ardy's inequality}.
\newblock Universitext. Springer, Cham, 2015.

\bibitem{BV}
H.~Brezis and J.~L. V{\'a}zquez.
\newblock Blow-up solutions of some nonlinear elliptic problems.
\newblock {\em Rev. Mat. Univ. Complut. Madrid}, 10(2):443--469, 1997.

\bibitem{CKN}
L.~Caffarelli, R.~Kohn, and L.~Nirenberg.
\newblock First order interpolation inequalities with weights.
\newblock {\em Compositio Math.}, 53(3):259--275, 1984.

\bibitem{Calogero}
F.~Calogero.
\newblock Ground state of a one‐dimensional n‐body system.
\newblock {\em Journal of Mathematical Physics}, 10(12):2197--2200, 1969.

\bibitem{CRT}
\'{O}. Ciaurri, L.~Roncal, and S.~Thangavelu.
\newblock Hardy-type inequalities for fractional powers of the
  {D}unkl-{H}ermite operator.
\newblock {\em Proc. Edinb. Math. Soc. (2)}, 61(2):513--544, 2018.

\bibitem{DX}
F.~Dai and Y.~Xu.
\newblock {\em Analysis on {$h$}-harmonics and {D}unkl transforms}.
\newblock Advanced Courses in Mathematics. CRM Barcelona.
  Birkh\"auser/Springer, Basel, 2015.
\newblock Edited by Sergey Tikhonov.

\bibitem{Davies1998}
E.~B. Davies.
\newblock A review of {H}ardy inequalities.
\newblock In {\em The {M}az'ya anniversary collection, {V}ol. 2 ({R}ostock,
  1998)}, volume 110 of {\em Oper. Theory Adv. Appl.}, pages 55--67.
  Birkh\"{a}user, Basel, 1999.

\bibitem{Dunkl1989}
C.~F. Dunkl.
\newblock Differential-difference operators associated to reflection groups.
\newblock {\em Trans. Amer. Math. Soc.}, 311(1):167--183, 1989.

\bibitem{FT}
S.~Filippas and A.~Tertikas.
\newblock Optimizing improved {H}ardy inequalities.
\newblock {\em J. Funct. Anal.}, 192(1):186--233, 2002.

\bibitem{GR}
M.~Ghergu and V.~D. R\u{a}dulescu.
\newblock {\em Nonlinear {PDE}s}.
\newblock Springer Monographs in Mathematics. Springer, Heidelberg, 2012.
\newblock Mathematical models in biology, chemistry and population genetics,
  With a foreword by Viorel Barbu.

\bibitem{GIT}
D.~V. Gorbachev, V.~I. Ivanov, and S.~Yu. Tikhonov.
\newblock Sharp {P}itt inequality and logarithmic uncertainty principle for
  {D}unkl transform in {$L^2$}.
\newblock {\em J. Approx. Theory}, 202:109--118, 2016.

\bibitem{GIT17}
D.~V. Gorbachev, V.~I. Ivanov, and S.~Yu. Tikhonov.
\newblock Riesz potential and maximal function for dunkl transform.
\newblock {\em arXiv preprint arXiv:1708.09733}, 2017.

\bibitem{HOHOLT}
M.~Hoffmann-Ostenhof, T.~Hoffmann-Ostenhof, A.~Laptev, and J.~Tidblom.
\newblock Many-particle {H}ardy inequalities.
\newblock {\em J. Lond. Math. Soc. (2)}, 77(1):99--114, 2008.

\bibitem{Ilin}
V.~P. Il'in.
\newblock Some integral inequalities and their applications in the theory of
  differentiable functions of several variables.
\newblock {\em Mat. Sb. (N.S.)}, 54 (96):331--380, 1961.

\bibitem{Lundholm}
D.~Lundholm.
\newblock Geometric extensions of many-particle {H}ardy inequalities.
\newblock {\em J. Phys. A}, 48(17):175203, 25, 2015.

\bibitem{LundholmSolovej}
D.~Lundholm and J.~P. Solovej.
\newblock Hardy and {L}ieb-{T}hirring inequalities for anyons.
\newblock {\em Comm. Math. Phys.}, 322(3):883--908, 2013.

\bibitem{Mazya}
V.~Maz'ya.
\newblock {\em Sobolev spaces with applications to elliptic partial
  differential equations}, volume 342 of {\em Grundlehren der Mathematischen
  Wissenschaften [Fundamental Principles of Mathematical Sciences]}.
\newblock Springer, Heidelberg, augmented edition, 2011.

\bibitem{Mejjaoli1}
H.~Mejjaoli.
\newblock Generalized homogeneous {B}esov spaces and their applications.
\newblock {\em Serdica Math. J.}, 38(4):575--614, 2012.

\bibitem{Mejjaoli2}
H.~Mejjaoli.
\newblock Generalized {L}orentz spaces and applications.
\newblock {\em J. Funct. Spaces Appl.}, pages Art. ID 302941, 14, 2013.

\bibitem{OP76}
M.~A. Olshanetsky and A.~M. Perelomov.
\newblock Completely integrable {H}amiltonian systems connected with semisimple
  {L}ie algebras.
\newblock {\em Invent. Math.}, 37(2):93--108, 1976.

\bibitem{OP78}
M.~A. Olshanetsky and A.~M. Perelomov.
\newblock Quantum systems related to root systems and radial parts of {L}aplace
  operators.
\newblock {\em Funct. Anal. Appl.}, 12(2):121--128, 1978.

\bibitem{OK}
B.~Opic and A.~Kufner.
\newblock {\em Hardy-type inequalities}, volume 219 of {\em Pitman Research
  Notes in Mathematics Series}.
\newblock Longman Scientific \& Technical, Harlow, 1990.

\bibitem{Rellich0}
F.~Rellich.
\newblock Halbbeschr\"{a}nkte {D}ifferentialoperatoren h\"{o}herer {O}rdnung.
\newblock In {\em Proceedings of the {I}nternational {C}ongress of
  {M}athematicians, 1954, {A}msterdam, vol. {III}}, pages 243--250. Erven P.
  Noordhoff N.V., Groningen; North-Holland Publishing Co., Amsterdam, 1956.

\bibitem{Rellich}
F.~Rellich.
\newblock {\em Perturbation theory of eigenvalue problems}.
\newblock Assisted by J. Berkowitz. With a preface by Jacob T. Schwartz. Gordon
  and Breach Science Publishers, New York-London-Paris, 1969.

\bibitem{Rosler2000}
M.~R\"{o}sler.
\newblock Short-time estimates for heat kernels associated with root systems.
\newblock In {\em Special functions ({H}ong {K}ong, 1999)}, pages 309--323.
  World Sci. Publ., River Edge, NJ, 2000.

\bibitem{Rosler}
M.~R\"{o}sler.
\newblock Dunkl operators: theory and applications.
\newblock In {\em Orthogonal polynomials and special functions ({L}euven,
  2002)}, volume 1817 of {\em Lecture Notes in Math.}, pages 93--135. Springer,
  Berlin, 2003.

\bibitem{Sutherland}
B.~Sutherland.
\newblock Quantum many‐body problem in one dimension: Ground state.
\newblock {\em Journal of Mathematical Physics}, 12(2):246--250, 1971.

\bibitem{Tidblom}
J.~Tidblom.
\newblock {\em Improved $L^p$ {H}ardy inequalities}.
\newblock PhD thesis, University of Stockholm, 2005.

\bibitem{vDV}
J.~F. van Diejen and L.~Vinet, editors.
\newblock {\em Calogero-{M}oser-{S}utherland models}, CRM Series in
  Mathematical Physics. Springer-Verlag, New York, 2000.

\bibitem{V}
A.~Velicu.
\newblock Sobolev-type inequalities for {D}unkl operators.
\newblock {\em arXiv preprint arXiv:1811.11118}, 2018.

\end{thebibliography}

\end{document}